\documentclass[12pt]{article}
\usepackage[T1]{fontenc}
\usepackage[utf8]{inputenc}
\usepackage{lmodern}

\usepackage[margin=1in]{geometry}
\usepackage{amsfonts,amssymb,amsthm,mathtools,amsmath}
\usepackage{enumitem}
\usepackage{microtype}
\usepackage[hidelinks]{hyperref}

\usepackage{tikz}
\usetikzlibrary{positioning,decorations.pathreplacing}

\newtheorem{theorem}{Theorem}[section]
\newtheorem{lemma}[theorem]{Lemma}
\newtheorem{proposition}[theorem]{Proposition}
\newtheorem{corollary}[theorem]{Corollary}
\newtheorem{conjecture}[theorem]{Conjecture}
\newtheorem{question}[theorem]{Question}
\theoremstyle{definition}
\newtheorem{definition}[theorem]{Definition}
\theoremstyle{remark}
\newtheorem{remark}[theorem]{Remark}

\newcommand{\ex}{\operatorname{ex}}
\newcommand{\exbip}{\operatorname{ex}_{\mathrm{bip}}}
\newcommand{\exor}{\operatorname{ex}^{\rightarrow}_{\mathrm{bip}}}

\newcommand{\inc}{\operatorname{inc}}

\title{Bipartite Extremal Numbers of Trees}

\author{
Lucas Waite\thanks{
Department of Mathematics, Kenyon College, Gambier, OH 43022, USA.
Email: waite1@kenyon.edu.
}
\and
Nuh Aydin\thanks{
Department of Mathematics, Kenyon College, Gambier, OH 43022, USA.
Email: aydinn@kenyon.edu.
}
}

\date{July 2026}

\begin{document}

\maketitle

\begin{abstract}
We study a restriction of the classical Erd\H{o}s--S\'os problem, the extremal number of trees, to the class of bipartite host graphs, both when only the order of the host is prescribed and when its two part-sizes are fixed. We give natural lower-bound constructions and formulate corresponding linear upper-bound conjectures. We apply a weighted variant of $k$-minimality to prove upper bounds for a broad family of trees including brooms, trees with part-sizes obeying certain inequalities, and all trees on at most seven vertices, resolving part of a problem of Caro, Patk\'os and Tuza up to additive constants. We also relate the fixed-part extremal number of a tree to the ordinary extremal number, and consider an oriented bipartite extremal function analogous to the Zarankiewicz function.
\end{abstract}


\section{Introduction}

Extremal graph theory asks how many edges a graph may contain while avoiding a prescribed subgraph. For a graph $H$, let $\ex(N;H)$ denote the maximum number of edges in an $H$-free graph on $N$ vertices. When $H$ is nonbipartite, the asymptotic behavior of $\ex(N;H)$ is determined by the chromatic number of $H$. The Erd\H{o}s--Stone--Simonovits theorem states that
\[
\ex(N;H)=\left(1-\frac{1}{\chi(H)-1}+o(1)\right)\frac{N^2}{2}
\]
whenever $\chi(H)\geq 3$; see, for example,
\cite{erdosStone1946,simonovits1968}. When $H$ is bipartite, however, this theorem gives only
\[
\ex(N;H)=o(N^2),
\]
and the correct order of magnitude can depend on detailed structural properties of $H$.

One of the foundational results in bipartite extremal graph theory is the K\H{o}v\'ari--S\'os--Tur\'an theorem, which gives
\[
\ex(N;K_{s,t})=O_{s,t}\left(N^{2-1/s}\right)
\]
for fixed integers $2\leq s\leq t$; see \cite{kovari1954}. Since every fixed bipartite graph is contained in some complete bipartite graph, this also gives a general subquadratic upper bound for every fixed bipartite forbidden graph. The related Zarankiewicz problem asks for the maximum number of edges in a bipartite graph with prescribed part-sizes that avoids a fixed complete bipartite graph.

A key unresolved question in extremal graph theory concerns the extremal numbers of trees. If $T$ is a tree on $t$ vertices, the Erd\H{o}s--S\'os conjecture predicts the optimal linear upper bound.

\begin{conjecture}[Erd\H{o}s--S\'os]
\label{conj:erdos-sos}
If $T$ is a tree on $t$ vertices, then every $T$-free graph $G$ on $N$ vertices satisfies
\[
e(G)\leq \frac{t-2}{2}N.
\]
\end{conjecture}

This conjecture was introduced by Erd\H{o}s \cite{erdos1964extremal}; related results can be found in
\cite{sidorenko2018keyrings,goerlich2016,besomi2021}.
Extremal problems for trees in bipartite host graphs have also appeared in
\cite{erdos1964extremal,kopylov1977,yuan2017,li2021,chen2023,yuan2024}.
A systematic study of bipartite Tur\'an numbers for trees was initiated by Caro, Patk\'os and Tuza \cite{caro2025}.

In this paper, we study the Erd\H{o}s--S\'os problem under the additional assumption that the host graph is bipartite. For a bipartite graph $H$, let $\exbip(N;H)$ denote the maximum number of edges in an $H$-free bipartite graph on $N$ vertices. We also study the finer fixed-part parameter
\[
\exbip(m,n;H),
\]
defined to be the maximum number of edges in an $H$-free bipartite graph with part-sizes $m$ and $n$.

Every tree is bipartite. We write $T_{a,b}=(A,B)$ when $A$ and $B$ are the two parts of a tree $T$, with
\[
|A|=a,\qquad |B|=b,\qquad a\geq b.
\]
The imbalance between $a$ and $b$ plays an important role in the bipartite-host problem. In particular, there are two natural types of lower-bound construction. Comparing these constructions suggests a transition at $a=2b-1$. More precisely, the constructions in Section~\ref{sec:constructions} motivate the following two conjectures.

\begin{conjecture}[Order form]
\label{conj:order}
For every tree $T_{a,b}$,
\[
    \exbip(N;T_{a,b})\le\begin{cases}
        \frac12(a-1)N& \text{if}\ a\geq2b-1,\\
        (b-1)N& \text{if}\ a\leq2b-1.
    \end{cases}
\]
\end{conjecture}

\begin{conjecture}[Fixed-part form]
\label{conj:fixed-parts}
For every tree $T_{a,b}$ and all $n\geq m$,
\[
    \exbip(m,n;T_{a,b})\le(b-1)n+\begin{cases}
    (a-b)m&\text{if}\ a\geq2b-1,\\
    (b-1)m&\text{if}\ a\leq2b-1.
    \end{cases}
\]
\end{conjecture}

Our main result proves Conjecture~\ref{conj:fixed-parts} for a broad family of trees. To state it, we introduce a parameter measuring the concentration of simple pendant branches around a vertex in the smaller part of the tree. For a tree $T$ and a vertex $x\in V(T)$, let $N_T(x)$ denote the set of vertices  of $T$ adjacent to $x$. Define
\[
    \mathcal S_T(x):=\{y\in N_T(x):N_T(y)\setminus\{x\}\text{ consists entirely of leaves}\},
\]
and for $T_{a,b}=(A,B)$, define $\sigma(T_{a,b}):=\max_{x\in B}|\mathcal S_T(x)|$.  Thus, the vertices in $\mathcal S_T(x)$ are precisely the leaves and pendant preleaves adjacent to $x$.

Our main theorem is the following.

\begin{theorem}\label{thm:main}
Let $T_{a,b}=(A,B)$ be a tree with $a\ge b$. If $\sigma(T_{a,b})\ge \min\{a-b,b-1\},$ then for every $n\geq m$,
\[
    \exbip(m,n;T_{a,b})\le(b-1)n+\begin{cases}
    (a-b)m&\text{if}\ a\geq2b-1,\\
    (b-1)m&\text{if}\ a\leq2b-1.
    \end{cases}
\]
\end{theorem}

Theorem~\ref{thm:main} applies, in particular, to stars, paths, double-stars, brooms, all trees whose part-sizes $a$ and $b$ obey any of $a\leq b+1$, $a\geq b(b-1)+1$ or $b\leq 2$, and all trees on at most seven vertices. Here, a \emph{broom} is a tree obtained from a vertex-disjoint star and a path by joining the center of the star to a leaf of the path, and a \emph{double-star} is a tree obtained by joining the centers of two stars. Caro, Patk\'os and Tuza asked for the determination of the fixed-part bipartite extremal numbers of all trees on seven vertices \cite{caro2025}. Our theorem settles the ordinary version of this problem up to an additive constant.

The main tool used to prove Theorem~\ref{thm:main} is a weighted version of the standard minimal-subgraph argument. A familiar method for embedding a tree is to pass from a graph of sufficiently large average degree to an inclusion-minimal subgraph satisfying the relevant edge inequality. Minimality then forces a lower bound on the minimum degree, since deleting a vertex of insufficient degree would preserve the edge inequality.

We conclude the paper with two related directions. First, we compare the fixed-part bipartite extremal number with the ordinary extremal number by taking a random bipartite subgraph. Second, we introduce an oriented fixed-part parameter
\[
\exor(m,n;H_{a,b}),
\]
in which the two parts of the forbidden graph are required to embed into prescribed parts of the host. When $H_{a,b}$ is complete bipartite, this is the usual Zarankiewicz function. For every tree $T_{a,b}$, we prove upper and lower bounds for $\exor(m,n;T_{a,b})$ whose difference is at most $(a-1)(b-1)$.

\section{Notation}
\label{sec:notation}

We use primarily standard notation. All graphs considered in this paper are undirected, finite, and simple. For a graph $G$, write $d_G(v)$ for the degree of $v$, and write $\delta(G)$ and $\Delta(G)$ for its minimum and maximum degree. If $G=(U,V)$ is bipartite, define
\[
    \delta_G(U):=\min_{u\in U}d_G(u),
    \qquad
    \Delta_G(U):=\max_{u\in U}d_G(u),
\]
and similarly for $V$. We  say $G$ has \emph{part-sizes} $|U|$ and $|V|$. For a graph $G$ and a vertex subset $S\subseteq V(G)$ we write $G-S$ for the induced graph $G[V(G)\setminus S]$ and $N_G(S)$ for the set of vertices adjacent to some vertex in $S$. A \emph{leaf} is a vertex of degree one. Throughout the paper, we use subscript notation $T_{a,b}$ or $H_{a,b}$ to indicate a bipartite graph with part-sizes $a\geq b$. In the case of $T_{a,b}$, the graph is also assumed to be a tree. The symbol $\dot\bigcup$ indicates a disjoint union.
\section{Constructions}
\label{sec:constructions}

We begin with some constructions to lower bound our extremal parameters. We note that these bounds are not tight in full generality, but appear to be tight in terms of the linear coefficients of $N,m$ and $n$. The first applies when only the total order of the host is prescribed, while the second applies when its part-sizes are fixed.

\begin{proposition}[Basic lower bounds]
\label{prop:lower-bounds}
Let \(H_{a,b}\) be a connected bipartite graph with part-sizes \(a\geq b\) and \(a\geq2\). Write $N=2q(a-1)+r$ where $0\leq r<2(a-1)$ and $N\geq b-1$. Then

\[
\exbip(N;H_{a,b}) \geq
\begin{cases}
q(a-1)^2+\left\lfloor\frac{r^2}{4}\right\rfloor &\text{if } q(a-1)(a-2b+1)+\left\lfloor\frac{r^2}{4}\right\rfloor+(b-1)^2\geq (b-1)r,\\[3mm]
(b-1)(N-b+1)&\text{otherwise}.
\end{cases}
\]

\noindent Now suppose that \(n\geq m\geq b\), and write $m-b+1=q(a-1)+r$ where $0\leq r<a-1$. Then

\[
\exbip(m,n;H_{a,b})\ge
\begin{cases}
(b-1)(n-m+b-1)+q(a-1)^2+r^2
&
\begin{aligned}[t]
\text{if }&
q(a-1)(a-2b+1)\\
&+(r-b+1)^2\ge0,
\end{aligned}
\\[3mm]
(b-1)(m+n-2b+2)
&
\text{otherwise.}
\end{cases}
\]
\end{proposition}

\begin{proof}
We first consider $\exbip(N;H_{a,b})$:
    Let $G=K_{b-1,N-b+1}$. One part of $G$ has size $b-1$, so $G$ cannot contain $H_{a,b}$. Hence $e(G)=(b-1)(N-b+1).$ Now take the disjoint union of $\left\lfloor\frac{N}{2(a-1)}\right\rfloor$ copies of $K_{a-1,a-1}$ and a balanced complete bipartite graph on the remaining $r$ vertices. Every component has both parts of size at most $a-1$, so the resulting graph is $H_{a,b}$-free. Its number of edges is
    \[
        \left\lfloor\frac{N}{2(a-1)}\right\rfloor(a-1)^2
        +\left\lfloor\frac{r^2}{4}\right\rfloor.
    \]
    The stated lower bound on $\exbip(N;H_{a,b})$ comes from maximizing over these two cases.

\noindent We now move to $\exbip(m,n;H_{a,b})$:
    Let $G=K_{b-1,m-b+1}\,\dot\bigcup\,K_{b-1,n-b+1}$, placing the two components in opposite orientations across the prescribed bipartition:
    
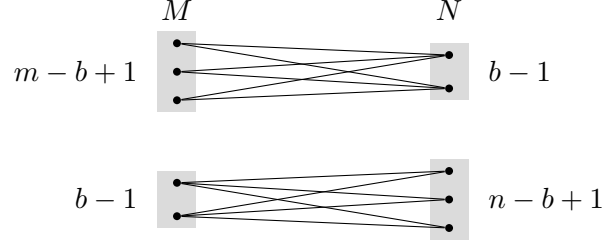
\begin{figure}[ht]
    \centering
    \begin{tikzpicture}[
        x=0.9cm,
        y=0.65cm,
        vertex/.style={
            circle,
            fill=black,
            inner sep=1.05pt
        },
        edge/.style={
            line width=0.4pt
        },
        block/.style={
            fill=black!14,
            draw=none
        },
        every node/.style={
            font=\small
        }
    ]

    \path[block] (-0.28,0.48) rectangle (0.28,2.12);
    \path[block] (3.72,0.72) rectangle (4.28,1.88);

    \path[block] (-0.28,-1.88) rectangle (0.28,-0.72);
    \path[block] (3.72,-2.12) rectangle (4.28,-0.48);

    \node[vertex] (u1) at (0,1.88) {};
    \node[vertex] (u2) at (0,1.30) {};
    \node[vertex] (u3) at (0,0.72) {};

    \node[vertex] (v1) at (4,1.64) {};
    \node[vertex] (v2) at (4,0.96) {};

    \node[vertex] (u4) at (0,-0.96) {};
    \node[vertex] (u5) at (0,-1.64) {};

    \node[vertex] (v3) at (4,-0.72) {};
    \node[vertex] (v4) at (4,-1.30) {};
    \node[vertex] (v5) at (4,-1.88) {};

    \foreach \i in {1,2,3}
        \foreach \j in {1,2}
            \draw[edge] (u\i) -- (v\j);

    \foreach \i in {4,5}
        \foreach \j in {3,4,5}
            \draw[edge] (u\i) -- (v\j);

    \foreach \i in {1,2,3,4,5}
        \fill (u\i) circle (1.05pt);

    \foreach \i in {1,2,3,4,5}
        \fill (v\i) circle (1.05pt);

    \node at (0,2.55) {$M$};
    \node at (4,2.55) {$N$};

    \node[anchor=east] at (-0.42,1.30) {$m-b+1$};
    \node[anchor=west] at (4.42,1.30) {$b-1$};

    \node[anchor=east] at (-0.42,-1.30) {$b-1$};
    \node[anchor=west] at (4.42,-1.30) {$n-b+1$};

    \end{tikzpicture}

    \caption{The construction
    \(G=K_{b-1,m-b+1}\,\dot\cup\,
    K_{b-1,n-b+1}\), with the two components oriented oppositely
    across the prescribed bipartition \(M\cup N\). Each shaded block
    schematically represents the indicated number of vertices.}
    \label{fig:fixed-part-lower-bound}
\end{figure}
    
    Each component has a part of size $b-1$, so $G$ is $H_{a,b}$-free. Moreover,
    \[
        e(G)=(b-1)(m-b+1)+(b-1)(n-b+1)=(b-1)(m+n-2b+2).
    \]
    Now for the other side of this bound take a component $K_{b-1,n-m+b-1}$, oriented so that its part of order $n-m+b-1$ lies in the part of $G$ of order $n$. On the remaining $m-b+1$ vertices of each side, take $q$ disjoint copies of $K_{a-1,a-1}$ and one copy of $K_{r,r}$. Every component either has a part of size $b-1$ or has both parts of size at most $a-1$, and hence the graph is $H_{a,b}$-free. Counting edges gives
    \[
        (b-1)(n-m+b-1)+q(a-1)^2+r^2.
    \]
    The stated lower bound on $\exbip(m,n;H_{a,b})$ again comes from maximizing over these two cases.
\end{proof}

Observe that in the above constructions dominance is exchanged asymptotically at $a=2b-1$. The preceding constructions establish that the bounds proposed in Conjectures~\ref{conj:order} and \ref{conj:fixed-parts} are sharp up to additive constants depending only on the forbidden tree. We do not expect the bounds in Conjectures~\ref{conj:order} and \ref{conj:fixed-parts} to be exact in general.

\section{Main Results}
\label{sec:main}
The key tool that we use in proving the main result is a weighted version of the familiar $k$-minimality.

\begin{definition}
Let $\mathcal P$ be a graph property. A graph $G$ is \emph{$\mathcal P$-minimal} if $G$ has property $\mathcal P$ but no proper subgraph of $G$ has property $\mathcal P$.
\end{definition}

The following lemma is a general statement of a commonly used argument in graph theory.
\begin{lemma}[Finite descent]\label{lem:finite-descent}
If a finite graph $G$ has a property $\mathcal P$, then $G$ contains a $\mathcal P$-minimal subgraph.
\end{lemma}
\begin{proof}
Among all subgraphs of $G$ having property $\mathcal P$, choose one that is minimal under the partial order of inclusion. Such a subgraph is guaranteed to exist because $G$ is finite.
\end{proof}

\begin{definition}\label{def:weighted-minimality}
Let $1\le r\le s$ be integers. A bipartite graph $G=(U,V)$ has property \emph{$(r,s)$} if $|U|\ge |V|$ and
\[
    e(G)>(r-1)|U|+(s-1)|V|.
\]
We say such graphs are \emph{$(r,s)$}-dense.
A bipartite graph $G$ is therefore \emph{$(r,s)$-minimal} if it is $(r,s)$-dense and no proper subgraph is $(r,s)$-dense after its two inherited parts are relabeled so that the larger part comes first.
\end{definition}

\noindent For a graph $G$ and a vertex subset $S\subseteq V(G)$, we  let
\[
    \inc_G(S):=\bigl|\{e\in E(G):e\cap S\ne\emptyset\}\bigr|
\]
denote the number of edges incident with at least one vertex of $S$. Thus, $e(G-S)=e(G)-\inc_G(S)$. We begin with some useful properties of $(r,s)$-minimal graphs.
\begin{lemma}
\label{lem:incidence}
Let $G=(U,V)$ be $(r,s)$-minimal where $r\leq s$. If $X\subseteq U$ and $Y\subseteq V$ satisfy $|X|\le |Y|+1$ and $X\cup Y\not=\emptyset$, then
\[
    \inc_G(X\cup Y)>(r-1)|X|+(s-1)|Y|.
\]
\end{lemma}

\begin{proof}
Suppose instead that $\inc_G(X\cup Y)\le (r-1)|X|+(s-1)|Y|$. Set $U'=U\setminus X$, $V'=V\setminus Y$, and $G'=G-(X\cup Y)$. We show that $G'$ contradicts the minimality of $G$ by considering two cases.

\noindent If $|U|-|V|\geq |X|-|Y|$, then
\begin{align*}
    e(G')&>(r-1)\bigl(|U|-|X|\bigr)+(s-1)\bigl(|V|-|Y|\bigr)\\
    &=(r-1)|U'|+(s-1)|V'|,
\end{align*}
and since
\[
|U'|=|U|-|X|\geq|V|-|Y|=|V'|,
\]
$G'$ is therefore $(r,s)$-dense.

\noindent If instead $|U|-|V|<|X|-|Y|\leq1$, then we must have $|X|=|Y|+1$ and $|U|=|V|$, so
\begin{align*}
    e(G')&>(r-1)\bigl(|U|-|X|\bigr)+(s-1)\bigl(|V|-|Y|\bigr)\\
    &=(r-1)\bigl(|V|-|X|\bigr)+(s-1)\bigl(|U|-|Y|\bigr)\\
    &\geq(r-1)\bigl(|V|-|Y|\bigr)+(s-1)\bigl(|U|-|X|\bigr)\\
    &=(r-1)|V'|+(s-1)|U'|,
\end{align*}
where the second-to-last line follows since $s\geq r$ and $|X|\geq|Y|$ therefore $(s-r)(|X|-|Y|)\geq0$. We also have 
\[
|V'|=|V|-|Y|>|U|-|X|=|U'|,
\]
and $G'$ is therefore $(r,s)$-dense. 

In either case $G'$ is $(r,s)$-dense, thus $G$ has a proper $(r,s)$-dense subgraph. By way of a contradiction, we must have $\inc_G(X\cup Y)>(r-1)|X|+(s-1)|Y|$.
\end{proof}

\begin{corollary}\label{cor:degrees}
If $G=(U,V)$ is $(r,s)$-minimal and $r\leq s$, then
\[
    \delta_G(U)\ge r,
    \qquad
    \delta_G(V)\ge s,
    \qquad
    \Delta_G(V)\ge r+s-1.
\]
\end{corollary}

\begin{proof}
Applying Lemma~\ref{lem:incidence} with $X=\{u\}$ and $Y=\emptyset$ gives $d_G(u)>r-1$ for every $u\in U$, and hence $\delta_G(U)\ge r$. Similarly, taking $X=\emptyset$ and $Y=\{v\}$ yields $\delta_G(V)\ge s$. Lastly,
\[
    e(G)>(r-1)|U|+(s-1)|V|
    \ge (r+s-2)|V|.
\]
The average degree of a vertex in $V$ is therefore greater than $r+s-2$, so some vertex of $V$ has degree at least $r+s-1$.
\end{proof}

\noindent We have the following embedding lemma that shows we can embed a large class of trees satisfying a particular condition.

\begin{lemma}
\label{lem:embedding}
Let $T_{a,b}=(A,B)$ be a tree with $a\ge b$, and let $G=(U,V)$ be bipartite. Suppose
\[
    \delta_G(U)\ge b\qquad\text{and}\qquad\delta_G(V)\ge a-\sigma(T_{a,b}).
\]
Choose $t\in B$ with $|\mathcal S_T(t)|=\sigma(T_{a,b})$. If $v\in V$ satisfies $d_G(v)\ge a$, then $G$ contains an embedding of $T_{a,b}$ that maps $t$ to $v$.
\end{lemma}

\begin{proof}
Set $\varphi(t)=v$. Define 
\[
T'=T_{a,b}-(\mathcal{S}_{T_{a,b}}(t)\cup (N_T(\mathcal{S}_{T_{a,b}}(t))\setminus\{t\})).
\]
That is, $T'$ is the subtree of $T_{a,b}$ with the leaves and `leaf-attachments' adjacent to $t$ removed.  We write $T'=(A',B')$ where $A'\subseteq A$ and $B'\subseteq B$.  We have 
\[
|A'|=a-\sigma(T_{a,b})\leq\delta_G(V),\quad \text{and}\quad|B'|\leq b\leq\delta_G(U),
\]
thus we may embed $T'$ greedily from $v$, embedding vertices of $A'$ to $U$ and consequently vertices of $B'$ to $V$.

We now extend $T'$ to $T_{a,b}$: $v$ has degree at least $a$, so as $T_{a,b}$ has $|A|=a$, we may embed each remaining neighbor of $t$ into $U$ greedily. We have thus embedded all of $A$ into $G$. The remaining vertices of $B$ are all leaves. Since $\delta_G(U)\geq b$ and $|B|=b$, they may be embedded greedily one at a time. We thus construct an embedding $\varphi$ of $T_{a,b}$ in $G$ with $\varphi(t)=v$ as desired.
\end{proof}

\noindent We now combine the minimality method with Lemma~\ref{lem:embedding} to prove the main theorem of the paper.

\begin{proof}[Proof of Theorem~\ref{thm:main}]
Let $G=(U,V)$ be bipartite with $|U|=n\ge m=|V|$, and suppose
\[
    e(G)>
    (b-1)n+\begin{cases}
    (a-b)m&\text{if}\ a\geq2b-1,\\
    (b-1)m&\text{if}\ a\leq2b-1.
    \end{cases}
\]
We consider two cases.

If $a\geq2b-1$, then we have $ e(G)>(b-1)|U|+(a-b)|V|.$ By Lemma~\ref{lem:finite-descent}, $G$ contains a $(b,a-b+1)$-minimal subgraph $G'=(U',V')$. Corollary~\ref{cor:degrees} gives
\[
    \delta_{G'}(U')\ge b,
    \qquad
    \delta_{G'}(V')\ge a-b+1,
    \qquad
    \Delta_{G'}(V')\ge a.
\]
Since $\sigma(T_{a,b})\ge b-1$, we have $a-\sigma(T_{a,b})\le a-b+1$.
Lemma~\ref{lem:embedding} therefore embeds $T_{a,b}$ in $G'$ and hence in $G$.

If $a\leq2b-1$, then we have $ e(G)>(b-1)|U|+(b-1)|V|.$ By Lemma~\ref{lem:finite-descent}, $G$ contains a $(b,b)$-minimal subgraph $G'=(U',V')$. Corollary~\ref{cor:degrees} gives
\[
    \delta_{G'}(U')\ge b,
    \qquad
    \delta_{G'}(V')\ge b,
    \qquad
    \Delta_{G'}(V')\ge 2b-1\geq a.
\]
Since $\sigma(T_{a,b})\ge a-b$, we have $a-\sigma(T_{a,b})\le b$.
Lemma~\ref{lem:embedding} therefore embeds $T_{a,b}$ in $G'$ and hence in $G$. Thus every graph with more than the claimed number of edges contains $T_{a,b}$.
\end{proof}

Theorem \ref{thm:main} is strong enough to resolve Conjecture \ref{conj:fixed-parts} for a large class of trees, including stars, paths, double-stars, brooms, all trees with part-sizes within specific ranges and all trees on at most seven vertices. Caro, Patk\'os and Tuza asked for the determination of both the ordinary and connected fixed-part bipartite extremal numbers for all trees on seven vertices. Our theorem settles the ordinary problem up to an additive constant. We write these well-known classes as a corollary:

\begin{proposition}
\label{prop:small-trees}
Theorem~\ref{thm:main} applies to every tree $T_{a,b}$ satisfying
$b\leq 2$, $a\leq b+1$ or $a\geq b(b-1)+1$, every double-star, every broom, and every tree on at most seven vertices.
\end{proposition}

\begin{proof}

If $a=b$, then the condition in Theorem~\ref{thm:main} is vacuous.
Suppose that $a>b$. Since
\[
    \sum_{ x\in A}d_{T_{a,b}}(x)=e(T_{a,b})=a+b-1<2a,
\]
some vertex of $A$ is a leaf. Hence $\sigma(T_{a,b})\geq 1$.
It follows that the hypothesis of Theorem~\ref{thm:main} holds
whenever $b\leq 2$ or $a-b\leq 1$. 

Suppose $a>b(b-1)$. We will show that some vertex in $B$ has at least $b-1$ leaves. Let $\ell(v)$ be the number of pendant vertices adjacent to $v$. Observing
\[
\sum_{x\in A}(d_{T_{a,b}}(x)-1)=e(T_{a,b})-a=b-1,
\]
there can be at most $b-1$ vertices $x\in A$ with $d_{T_{a,b}}(x)-1\geq 1$. The result is that there are at least $a-b+1$ vertices $x\in A$ with $d_{T_{a,b}}(x)\leq1$, so there are at least $a-b+1>(b-1)^2$ leaves in the $A$ part of $T_{a,b}$. By averaging, there exists a vertex $v\in B$ with
\[
\ell(v)>\frac{(b-1)^2}{b}=b-2+\frac{1}{b}
\]
leaves. Since this number must be an integer, we have $\ell(v)\geq b-1$ and therefore $\sigma(T_{a,b})\geq b-1$.

If $T_{a,b}$ is a double-star, then $\sigma(T_{a,b})=a\geq a-b$. If $T_{a,b}$ is a broom, then inspection of the bipartition along
its handle shows that $\sigma(T_{a,b})=a-b$ when there is a leaf in $B$ and $\sigma(T_{a,b})=a-b+1$ otherwise.

Finally, suppose that $a+b\leq 7$. If $b\leq 2$, the result has
already been proved. If $b\geq 3$, then $a-b\leq 1$, since otherwise
$a\geq b+2$ and $a+b\geq 2b+2\geq 8$, a contradiction.

Together these conditions yield the desired result.
\end{proof}

An immediate corollary of this result is that for each fixed integer $b$, Conjecture~\ref{conj:fixed-parts} holds for every tree $T_{a,b}$ with $a$ outside of the finite range $b+2\leq a\leq b(b-1)$. The smallest tree for which Theorem \ref{thm:main} does not apply is the tree constructed by taking a path on seven vertices and attaching a leaf to the center vertex. Expanding the class of trees for which we know Conjecture~\ref{conj:fixed-parts} holds is of primary interest to the authors of this paper.

\begin{corollary}
\label{cor:order-main}
Under the hypotheses of Theorem~\ref{thm:main},
\[
    \exbip(N;T_{a,b})\le\begin{cases}
        \frac12(a-1)N&\text{if}\ a\geq2b-1,\\
        (b-1)N&\text{if}\ a\leq2b-1.
    \end{cases}
\]
\end{corollary}

\begin{proof}
Let $G=(U,V)$ be bipartite on $N$ vertices, with $|U|=n\ge m=|V|$. If $a\le 2b-1$, the bound in Theorem~\ref{thm:main} is $(b-1)N$. If $a\ge 2b-1$, then
\[
(b-1)n+(a-b)m\le \frac{a-1}{2}(m+n)=\frac{a-1}{2}N,
\]
where the inequality follows from $n\ge m$ and $a-b\ge b-1$. Applying Theorem~\ref{thm:main} proves the result.
\end{proof}

\begin{remark}
\label{rmk:fixed-implies-order}
    The preceding argument shows that for any fixed tree $T_{a,b}$ if Conjecture~\ref{conj:fixed-parts} holds for $T_{a,b}$ then Conjecture~\ref{conj:order} holds for $T_{a,b}$.
\end{remark}

\noindent We also have the following resulting corollary for all trees:
\begin{corollary}
\label{cor:weak-general}
For any tree $T_{a,b}$ and any $n\geq m$,
\[
    \exbip(m,n;T_{a,b})\le(b-1)n+\begin{cases}
    (b-1)m&\text{if}\ a\leq2b-1,\ \sigma(T_{a,b})\geq a-b,\\
    (a-b)m&\text{if}\ a\geq2b-1,\ \sigma(T_{a,b})\geq b-1,\\
    (a-1-\sigma(T_{a,b}))m&\text{otherwise}.
    \end{cases}
\]
\end{corollary}
\begin{proof}
    If $\sigma(T_{a,b})\geq\min\{a-b,b-1\}$, then we apply Theorem \ref{thm:main}. 
    
    \noindent If instead $\sigma(T_{a,b})<\min\{a-b,b-1\}$, we construct a tree $T'$ by attaching $b-1-\sigma(T_{a,b})$ leaves to a vertex $x\in B$ achieving $|\mathcal S_{T_{a,b}}(x)|=\sigma(T_{a,b})$. Thus $T_{a,b}\subseteq T'$, $T'$ has parts $A'$ and $B'$ with $|A'|=a+b-1-\sigma(T_{a,b})$ and $|B'|=b$, and we have
    \[
    \sigma(T')\geq\sigma(T_{a,b})+(b-1-\sigma(T_{a,b}))=b-1.
    \]
    Writing $a':=|A'|=a+b-1-\sigma(T_{a,b})$ we have $\sigma(T_{a,b})<a-b$ and hence $a'\geq 2b$. Thus the first case of Theorem~\ref{thm:main} applies to $T'$. Since $a'-b=a-1-\sigma(T_{a,b})$, we obtain
    \[
    \exbip(m,n;T')\leq(b-1)n+(a-1-\sigma(T_{a,b}))m.
    \]
    Observing that graphs which are $T_{a,b}$-free must also be $T'$-free, we achieve
    \[
    \exbip(m,n;T_{a,b})\leq(b-1)n+(a-1-\sigma(T_{a,b}))m.
    \]
    Combining our two cases yields the desired result.
\end{proof}

\noindent Khormali \footnote{Personal communication of a preprint titled ``Bipartite Tur\'an Numbers of Trees and Star Forests'' by O. Khormali. The result cited here is Theorem 1.2 in his manuscript. It is now available at https://arxiv.org/pdf/2608.01873} has independently obtained, by a different method, that for
$b\geq 2$ and
\[
    n\geq a(m-b+2)\binom{m}{b-1},
\]
one has
\[
    \exbip(m,n;T_{a,b})\leq(b-1)n+\binom{m}{b-1}^2(a-1)+m\left\lfloor\frac{m}{b-1}\right\rfloor a+ma(m-b+2)\binom{m}{b-1}.
\]
For fixed $m$, both this bound and Corollary~\ref{cor:weak-general}
have the leading term $(b-1)n$. Our result additionally gives an explicit linear dependence on $m$. Moreover, we do not assume any non-trivial condition on $n$.

\section{Related Results}
\label{sec:related}

We obtain some related results, the first of which is a comparison of bipartite extremal number to the classical extremal number via a random bipartite subgraph.

\begin{proposition}
\label{prop:random-bipartition}
    For any tree $T_{a,b}$ and positive integers $n\geq m$, 
    \[
    \ex(m+n;T_{a,b})\leq \frac{(m+n)(m+n-1)}{2mn}\,\exbip(m,n;T_{a,b}).
    \]
\end{proposition}

\begin{proof}
    Consider a graph $G$ on $m+n$ vertices which is $T_{a,b}$-free and has $\ex(m+n;T_{a,b})$ edges. Randomly select $X\subseteq V(G)$ with $|X|=m$. Write $G[X,Y]$ for the subgraph of $G$ which consists of only the edges between $X$ and $Y=V(G)\setminus X$. Observe that $G[X,Y]$ is $T_{a,b}$-free. For any edge $uv\in E(G)$, \begin{align*}
        \Pr(uv\in G[X,Y])&=\Pr(u\in X,\ v\in Y)+\Pr(v\in X,\ u\in Y)\\
        &=2\Pr(u\in X,\ v\in Y)\text{ by symmetry.}\\
        &=2\binom{m+n-2}{m-1}/\binom{m+n}{m}\text{ by simple counting.}\\
        &=\frac{2mn}{(m+n)(m+n-1)}.
    \end{align*}
    
     Observe $e(G)\Pr(uv\in G[X,Y])=\mathbb{E}(e(G[X,Y]))$ by linearity of expectation, since there are $e(G)$ edges in $G$, each with $\Pr(uv\in G[X,Y])$ probability of being in $G[X,Y]$. So $G$ has a subgraph $G'$ with parts $m$ and $n$ such that 
     \[
     \exbip(m,n;T_{a,b})\geq e(G')\geq e(G)\Pr(uv\in G[X,Y])=\frac{2mn}{(m+n)(m+n-1)}\ex(m+n;T_{a,b}).
     \]
     Rearranging yields the desired bound.
\end{proof}

Even assuming Conjecture~\ref{conj:fixed-parts}, this random-bipartition comparison does not in general recover the conjectured Erd\H{o}s--S\'os bound for $T_{a,b}$. It is expected to be most useful when the two part-sizes of the tree are highly unbalanced. It would be of interest to tighten the relationship between these two parameters.

\medskip

We now move to discussing the oriented extremal parameter. Let $G=(M,N)$ be a bipartite graph with $|M|=m$ and $|N|=n$. Define the \textit{oriented bipartite extremal number} of $H$, $\exor(m,n;H_{a,b})$, to be the maximum number of edges in such a graph $G$ containing no copy of $H_{a,b}$ in which the $b$-vertex part is embedded in $M$ and the $a$-vertex part is embedded in $N$. When $H_{a,b}$ is complete bipartite, this is the usual Zarankiewicz function.

\begin{remark}
    If a bipartite graph $H_{a,b}$ has an automorphism, that is, an edge-preserving bijection on $V(H_{a,b})$, which switches the parts of $H_{a,b}$, then $\exbip(m,n;H_{a,b})=\exor(m,n;H_{a,b})$ for any $n\geq m$.
\end{remark}

\begin{proposition}
\label{prop:exor-LB}
        For any connected bipartite graph $H_{a,b}$, if $m\ge b-1$ and $n\ge a-1$, then
    \[
        \exor(m,n;H_{a,b})\ge\max\{(b-1)n,\ (a-1)m,\ (b-1)n+(a-1)m-2(a-1)(b-1)\}.
    \]
\end{proposition}
\begin{proof}
    Let $G_1=K_{b-1,n-a+1}\dot\cup K_{m-b+1,a-1},\ G_2=K_{b-1,n}$, and $G_3=K_{m,a-1}$
where for $G_1$ the part of size $b-1$ in the first component and the part of size $m-b+1$ in the second lie in $M$. We complete $G_2$ and $G_3$ with isolated vertices to have parts of size $m$ and $n$. Every component in these three graphs either has fewer than $b$ vertices in $M$ or fewer than $a$ vertices in $N$. Thus $G_1, G_2$ and $G_3$ each contain no copy of $H_{a,b}$ with the $b$-vertex part embedded in $M$. We count
    \[
        e(G_1)=(b-1)(n-a+1)+(a-1)(m-b+1),\quad e(G_2)=(b-1)n,\quad e(G_3)=(a-1)m.
    \]
Maximizing over the lower bounds yields the desired result.
\end{proof}

\begin{proposition}
\label{prop:exor-UB}
    For all trees $T_{a,b}$ and all $m\geq b-1$, $n\geq a-1$,
    \[
    \exor(m,n;T_{a,b})\leq(a-1)m+(b-1)n-(a-1)(b-1).
    \]
\end{proposition}
\begin{proof}
    Let $G=(M,N)$ be a bipartite graph such that
    \[
    e(G)>(a-1)m+(b-1)n-(a-1)(b-1).
    \]
    Say a subgraph $G'=(M',N')$ of $G$ with $M'\subseteq M$ and $N'\subseteq N$ has property $P$ if 
    \[
    e(G')>(a-1)|M'|+(b-1)|N'|-(a-1)(b-1),\quad |N'|\geq a-1\quad\text{and}\quad |M'|\geq b-1.
    \]
    
    \noindent Consider a $P$-minimal subgraph $G'=(M',N')$: If $|M'|= b-1$, then 
    \[
    e(G')>(a-1)|M'|+(b-1)|N'|-(a-1)(b-1)=(b-1)|N'|,
    \]
    which is impossible as $|M'|=b-1$. Thus $|M'|>b-1$ and similarly $|N'|>a-1$.

    Now if there were a vertex $v\in M'$ with $d_{G'}\leq a-1$, then
    \begin{align*}
        e(G'-\{v\})&\geq e(G')-(a-1)\\
        &>(a-1)(|M'|-1)+(b-1)|N'|-(a-1)(b-1),
    \end{align*}
    which would imply $G'-\{v\}$ has the property $P$, since $|M'|-1\geq b-1$. Since $G'$ is $P$-minimal, this is impossible and therefore $\delta_{G'}(M')\geq a$, and by a symmetric argument we have $\delta_{G'}(N')\geq b$. We thus embed $T_{a,b}$ in $G'\subseteq G$ greedily: vertices of $B$ are embedded in $M'$ while vertices of $A$ are embedded in $N'$. 
\end{proof}

Together with Proposition \ref{prop:exor-LB} we leave open only a gap of at most $(a-1)(b-1)$ between our lower and upper bounds for $\exor(m,n;T_{a,b})$ for each tree $T_{a,b}$.

\section{Conclusion}

We used a simple generalization of classical $k$-minimality to yield strong asymmetric degree conditions on the two sides of a bipartite graph. Theorem~\ref{thm:main} shows that these conditions are sufficient to embed trees with a vertex in the smaller part adjacent to sufficiently many leaves or pendant preleaves in its larger part. The result is a resolution of a part of a problem of Caro, Patk\'os and Tuza up to additive constants. We related this problem to the Erd\H{o}s--S\'os Conjecture. We then obtained bounds on a related problem generalizing the Zarankiewicz function. For the oriented parameter, we obtain upper and lower bounds differing by at most a constant term $(a-1)(b-1)$.

Several natural problems remain.

\begin{enumerate}[label=\textup{(\roman*)}]
    \item Determine whether Conjectures~\ref{conj:order} and \ref{conj:fixed-parts} hold for every tree, or for larger classes of trees. 

    \item Tighten the relationship between Conjecture \ref{conj:fixed-parts} and the Erd\H{o}s--S\'os Conjecture.

    \item Further investigate $k$-partite or connected analogues of this problem.

    \item Determine the optimal additive terms in Conjectures~\ref{conj:order} and \ref{conj:fixed-parts}.
\end{enumerate}

We would also like to raise the following question:
\begin{question}
    For every fixed tree $T_{a,b}$, do there exist positive integers $m_0,n_0>0$ such that for all $m\geq m_0$ and $n\geq n_0$,
    \[
    \exor(m,n;T_{a,b})=(b-1)n+(a-1)m-2(a-1)(b-1)?
    \]
\end{question}

\section*{Acknowledgments}

The authors thank Professor Omid Khormali for helpful discussions on this topic, reviewing an earlier version of this manuscript, and giving detailed feedback. We note that Khormali independently investigated the same problem using a distinct approach after the majority of our work was completed.

\section*{Declaration of generative AI and AI-assisted technologies in the manuscript preparation process}

During the preparation of this work the authors used ChatGPT-5.6 Sol in order to improve readability of the manuscript and put references in the required format. After using this tool/service, the authors reviewed and edited the content as needed and take full responsibility for the content of the published article.

\end{document}